\documentclass[11pt]{amsart}
\usepackage{amsthm, amsmath, amssymb, amsfonts, fancyhdr, verbatim, url}
\usepackage[margin = 1in]{geometry}
\usepackage{graphicx}
\usepackage{setspace}
\usepackage{tikz, pgfplots}
\usepackage{booktabs}
\usepackage{hyperref}

\hypersetup{
    pdfstartview={FitH}    
}

\usepackage{mathpazo}

\newtheorem*{claim*}{Claim}
\newtheorem*{fact*}{Fact}
\newtheorem{theorem}{Theorem}[section]
\newtheorem*{theorem*}{Theorem}
\newtheorem{proposition}[theorem]{Proposition}
\newtheorem{lemma}[theorem]{Lemma}
\newtheorem*{lemma*}{Lemma}
\newtheorem{corollary}[theorem]{Corollary}
\newtheorem{conjecture}[theorem]{Conjecture}

\theoremstyle{definition}
\newtheorem{definition}[theorem]{Definition}
\newtheorem{example}[theorem]{Example}

\newtheorem*{question*}{Question}

\theoremstyle{remark}
\newtheorem*{remark}{Remark}

\def\c{\mathcal}
\def\AA{\mathbb{A}}

\def\NN{\mathbb{N}}
\def\ZZ{\mathbb{Z}}

\newcommand{\abs}[1]{\left\lvert#1\right\rvert}

\newcommand{\floor}[1]{\left\lfloor #1 \right\rfloor}

\def\({\left(}
\def\){\right)}

\def\vphi{\varphi}

\def\<={\Leftarrow}
\def\=>{\Rightarrow}

\title{Constructing Numerical Semigroups of a Given Genus}
\author{Yufei Zhao}
\date{\today}
\address{Department of Mathematics, Massachusetts Institute of Technology, Cambridge, Massachusetts}
\email{yufei.zhao@gmail.com}

\begin{document}

\begin{abstract}
Let $n_g$ denote the number of numerical semigroups of genus $g$. Bras-Amor\'os conjectured that $n_g$ possesses certain Fibonacci-like properties. Almost all previous attempts at proving this conjecture were based on analyzing the semigroup tree. We offer a new, simpler approach to counting numerical semigroups of a given genus. Our method gives direct constructions of families of numerical semigroups, without referring to the generators or the semigroup tree. In particular, we give an improved asymptotic lower bound for $n_g$.
\end{abstract}

\maketitle

\section{Introduction}

A \emph{numerical semigroup} is a subset $\Lambda$ of the non-negative integers that is closed under addition, contains $0$, and has finite complement in $\NN_0$. The size of $\NN_0 \setminus \Lambda$ is called the \emph{genus}, denoted $g = g(\Lambda)$. The smallest nonzero element of $\Lambda$ is called its \emph{multiplicity}, denoted $m = m(\Lambda)$. The largest element of $\NN_0 \setminus \Lambda$ is called the \emph{Frobenius number}, denoted $f = f(\Lambda)$.

Let $n_g$ denote the number of numerical semigroups of genus $g$. The sequence $n_g$ appears as entry A007323 in the Sloane's On-line Encyclopedia of Integer Sequences \cite{Sloane}. After computing the first 50 terms of the sequence, Bras-Amor\'os \cite{B08} observed a Fibonacci-like behavior and made the following conjecture. Here $\vphi = \frac{1 + \sqrt 5}2$ is the golden ratio.

\begin{conjecture}[Bras-Amor\'os] \label{conj:B}
$\displaystyle \lim_{g \to \infty} \frac{n_g}{n_{g-1}} = \vphi$ 
\quad and \quad
$\displaystyle \lim_{g \to \infty} \frac{n_{g-1} + n_{g-2}}{n_g} = 1$.
\end{conjecture}

Note that the first claim implies the second. It was also conjectured \cite{B08} that $n_g \geq n_{g-1} + n_{g-2}$, although we will not discuss this conjecture in this paper. It is not even known whether the sequence $n_g$ is increasing.

Almost all previous bounds on $n_g$ were obtained using the semigroup tree \cite{B09, BB, Eli}. In this paper we offer a new approach to attacking the conjecture. Our method is arguably simpler than the semigroup tree method, as it provides direct constructions of numerical semigroups viewed as sets of integers, without referring to the generators.

Let $F_n$ denote the Fibonacci numbers, defined by $F_1=F_2=1$, $F_{n+2} = F_{n+1} + F_n$ for $n\geq 1$. To simplify our notation, let $F_n = 0$ for any $n \leq 0$ (though this is not the usual convention). Using the semigroup tree method, Bras-Amor\'os \cite{B09} showed that
\[
	2F_g \leq n_g \leq 1 + 3 \cdot 2^{g-3},
\]
improving a previous upper bound of $\frac{1}{g+1}\binom{2g}{g}$ obtained using Dyck paths \cite{BM}. Recent work by Elizalde \cite{Eli} improved these bounds again using the semigroup tree method. Elizalde's new bounds are expressed in terms of generating functions. All known upper bounds are quite weak, as the best one, by Elizalde, has order of growth like $2^g / \sqrt{\pi g}$, which is far from the asymptotic behavior implied by Conjecture \ref{conj:B}, from which one should expect that $\lim_{g \to \infty} n_g^{1/g} = \vphi$.

In this paper, we propose a method to tackle the following stronger version of Conjecture \ref{conj:B}.

\begin{conjecture} \label{conj:ratio}
As $g \to \infty$, $n_g \vphi^{-g}$ converges to a finite limit.
\end{conjecture}

See Table \ref{tab:t_g} and Figure \ref{fig:ratio-plot} at the end of this paper for some computed values and plots of $n_g \vphi^{-g}$. Our results imply that
\[
	\liminf_{g \to \infty} n_g \vphi^{-g} > 3.78.
\]
For comparison, Bras-Amor\'os' \cite{B09} lower bound $n_g \geq 2F_g$ gives $\liminf_{g \to \infty} n_g \vphi^{-g} \geq \frac{2}{\sqrt 5} > 0.894$, and Elizalde's \cite{Eli} lower bound\footnote{Elizalde's lower bound is given as $n_g \geq a_g$, where
\[
	\sum_{g \geq 1} a_g t^g = \frac{t(1-t^2-2t^3-3t^4+t^5+2t^6+3t^7+3t^8+t^9)}{(1+t)(1-t)(1-t-t^2)(1-t-t^3)(1-t^3-2t^4-2t^5-t^6)}.
\]
Expanding using partial fractions gives us $a_n = \(1 + \frac{2}{\sqrt 5} \) \vphi^n + o(\vphi^n)$.
} gives $\liminf_{g \to \infty} n_g \vphi^{-g} \geq 1 + \frac{2}{\sqrt 5} > 1.894$.
It seems that $n_g \vphi^{-g}$ increases with $g$, although we do not know of a proof. Using computed values of $n_g$ from \cite{B08}, we have $n_{50} = 101090300128$, so that $n_{50}\vphi^{-50} \approx 3.59$. Thus, our numerical bound is a significant improvement over previous results, and we expect the true value $\lim_{g \to \infty} n_g \vphi^{-g}$ to be very close to our lower bound.

Let us mention as an aside that a method for computing $n_g$ was given by Blanco and Puerto \cite{BP}, who converted the problem of counting the number of numerical semigroups of a given genus and multiplicity to a problem of counting lattice points in a polytope. This gives a polynomial-time algorithm for computing $n_g$. However, their paper did not not give bounds for $n_g$.

The intuition behind our method is that we can enumerate numerical semigroups $\Lambda$ by the ``prefix'' of $\Lambda \setminus \{0\} - m(\Lambda)$. This idea was inspired by recent work on counting subsets of $\{0, 1, \dots, n\}$ with a prescribed number of missing sums and missing differences \cite{Zhao:limit}.

We begin with a warm-up in Section \ref{sec:f<2m} where we consider numerical semigroups satisfying $f < 2m$. Section \ref{sec:f<3m} contains the main part of our analysis, in which we count numerical semigroups satisfying $f < 3m$. Finally, Section \ref{sec:f>3m} contains some observations on the number of numerical semigroups left out by our analysis.


\section{Numerical semigroups with $f < 2m$} \label{sec:f<2m}

For $a \leq b$, let $[a,b]$ denote the set $\{a, a+1, \dotsc, b\}$, and let $[a, \infty)$ denote the set $\{a, a+1, \dots\}$. 
The following result shows how to construct all the numerical semigroups satisfying $f(\Lambda) < 2m(\Lambda)$.

\begin{proposition}
Let $m$ be a positive integer. The collection of numerical semigroups $\Lambda$ with multiplicity $m$ and satisfying $f(\Lambda) < 2m$ is exactly the collection of sets of the form
\begin{equation} \label{eq:f<2m-form}
	\Lambda = \{0, m\} \cup S \cup [2m, \infty)
\end{equation}
where $S \subset [m+1, 2m-1]$.
\end{proposition}

\begin{proof}
Let $\Lambda$ be a numerical semigroup with multiplicity $m$ and satisfying $f(\Lambda) < 2m$. Then $[0, m] \cap \Lambda = \{0, m\}$. Since $f(\Lambda) < 2m$, we have $[2m, \infty) \subset \Lambda$. Therefore, $\Lambda$ must have the form \eqref{eq:f<2m-form}. The $\Lambda$ in \eqref{eq:f<2m-form} is indeed a numerical semigroup, since if $a$ and $b$ are nonzero elements of $\Lambda$, then $a, b \geq m$ so that $a + b \geq 2m$ and hence $a + b \in \Lambda$.
\end{proof}

Now we restrict the genus of $\Lambda$. This is equivalent to restricting the size of $S$ in \eqref{eq:f<2m-form}. To simplify notation, in this paper we treat $\binom{a}{b}$ as zero unless $0 \leq b \leq a$.

\begin{corollary} \label{cor:f<2m-num}
Let $m$ and $g$ be positive integers. The numerical semigroups $\Lambda$ with multiplicity $m$, genus $g$, and satisfying $f(\Lambda) < 2m$ are exactly those sets of the form \eqref{eq:f<2m-form} with $S \subset [m+1, 2m-1]$ and $\abs{S} = 2m - 2 - g$. The number of such numerical semigroups is exactly $\binom{m-1}{2m - 2 - g}$.
\end{corollary}

By varying $m$, we obtain the main result of this section.

\begin{proposition} \label{prop:f<2m}
For any positive integer $g$, the number of numerical semigroups $\Lambda$ with genus $g$ satisfying $f(\Lambda) < 2 m(\Lambda)$ is $F_{g+1}$.
\end{proposition}

\begin{proof}
From Corollary \ref{cor:f<2m-num}, there are exactly $\binom{m-1}{2m - 2 - g}$ numerical semigroups with multiplicity $m$, genus $g$, and satisfying $f(\Lambda) < 2m$. By summing over all $m$, we find that the number of numerical semigroups with genus $g$ and satisfying $f(\Lambda) < 2m$ is
\begin{equation} \label{eq:constr1}
	\sum_{m} \binom{m-1}{2m - 2 - g} = \sum_{m} \binom{m-1}{g - m + 1} = F_{g+1},
\end{equation}
where the sum is taken over all finitely many $m$ for which the summand is nonzero. The last step comes from following well-known identity which can be proven easily by induction,
\begin{equation} \label{eq:fib}
	\sum_{k \geq 0} \binom{n-k}{k} = F_{n+1}. \qedhere
\end{equation}
\end{proof}


\section{Numerical semigroups with $f < 3m$} \label{sec:f<3m}

\subsection{Counting by type}
Now let us consider numerical semigroups $\Lambda$ with $2 m(\Lambda) < f(\Lambda) < 3 m (\Lambda)$. For any $A \subset \ZZ$ and $b \in \ZZ$, we write $A+A = \{a_1 + a_2 : a_1, a_2 \in A\}$ and $A + b = b + A = \{a + b : a \in A\}$. For a positive integer $k$, let
\[
	\c A_k = \{A \subset [0, k-1] :  0 \in A \text{ and }  k \notin A + A \}.
\]
For example,
\begin{align*} 
	\c A_1 &= \{ \{0\} \} \\
	\c A_2 &= \{ \{0\} \} \\
	\c A_3 &= \{ \{0\}, \{0, 1\}, \{0, 2\} \} \\ 
	\c A_4 &= \{ \{0\}, \{0, 1\}, \{0, 3\} \} \\
	\c A_5 &= \{ \{0\}, \{0, 1\}, \{0, 1, 2\}, \{0, 1, 3\}, \{0, 2\}, \{0, 2, 4\}, \{0, 3\}, \{0, 3, 4\}, \{0, 4\} \} \\
	\c A_6 &= \{ \{0\}, \{0, 1\}, \{0, 1, 2\}, \{0, 1, 4\}, \{0, 2\}, \{0, 2, 5\}, \{0, 4\}, \{0, 4, 5\}, \{0, 5\} \}
\end{align*}
For any set in $\c A_k$, at most one element is included from $\{x, k-x\}$ for each $1 \leq x < k/2$, so that $\abs{\c A_k} = 3^{\floor{(k-1)/2}}$.

\begin{definition}
Let $\Lambda$ be a numerical semigroup with multiplicity $m$ and Frobenius number $f$, such that $2m < f < 3m$.
We say that $\Lambda$ has \emph{type} $(A; k)$, where $k < m$ is a positive integer and $A \in \c A_k$, if $k < m$, $f = 2m + k$, and $\Lambda \cap [m, m+k] = A + m$.
\end{definition}

\begin{example} \label{ex:type}
Let us find all the numerical semigroups $\Lambda$ with $m = 5$ and of type $(\{0, 2\}, 3)$. Since $m = 5$, we have $\Lambda \cap [0, 5] = \{0, 5\}$. Since $\Lambda$ has type $(\{0, 2\}, 3)$, we know that $\Lambda \cap [0, 8] = \{0\} \cup (m + A) = \{0, 5, 7\}$ and $f(\Lambda) = 2m+k = 13$, so that $13\notin \Lambda$ and $[14, \infty) \subset \Lambda$. Since $5, 7 \in \Lambda$, we must have $10, 12 \in \Lambda$ as well. In fact, these are the only restrictions on $\Lambda$. That is, the numerical semigroups $\Lambda$ with $m = 5$ and of type $(\{0, 2\}, 3)$ all have the form
\[
	\Lambda = \{0, 5, 7, 10, 12\} \cup B \cup [14, \infty)
\]
where $B$ is any subset of $\{9, 11\}$. Hence there are four such numerical semigroups, namely
\begin{align*}
	\{0, 5, 7, 10, 12\} &\cup [14, \infty), \\
	\{0, 5, 7, 9, 10, 12\} &\cup [14, \infty), \\
	\{0, 5, 7, 10, 11, 12\} &\cup [14, \infty), \\
	\text{and } \{0, 5, 7, 9, 10, 11, 12\} &\cup [14, \infty).
\end{align*}
\end{example}

\medskip

Note that every numerical semigroup with $2m < f < 3m$ has a unique type $(A, k)$, since $k = f - 2m$ and $A = \Lambda \cap [m, m+k] - m$. The idea is to count numerical semigroups by their genus and type. The following result gives the construction of the family of numerical semigroups of a given type. See Figure \ref{fig:type-construction} for an illustration.

\begin{proposition} \label{prop:construction}
Let $m$ and $k$ be positive integers with $m > k$, and let $A \in \c A_k$. Then the collection of numerical semigroups with multiplicity $m$ and type $(A; k)$ is exactly the collection of sets of the form (written as a disjoint union)
\begin{equation} \label{eq:constr2}
	\Lambda = \{0\} \cup (m + A) \cup (2m + (A + A)\cap[0, k]) \cup B \cup [2m + k + 1, \infty)
\end{equation}
where $B$ is any subset of $[m+k+1, 2m+k-1] \setminus (2m + A + A)$.
\end{proposition}

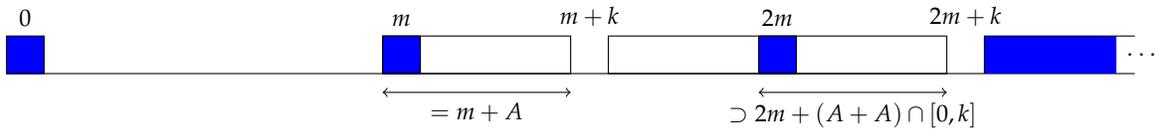
\begin{figure}[ht!]\centering

\begin{tikzpicture}[scale=.5, font={\footnotesize}]
	
	\draw (0,0) -- ++(30,0);
	\draw[fill=blue] (0,0) rectangle ++(1,1);
	\draw[fill=blue] (10,0) rectangle ++(1,1);
	\draw[fill=blue] (20,0) rectangle ++(1,1);
	\fill[blue] (26,0) rectangle (29.5,1);
	\draw (10,0) rectangle ++(5,1);
	\draw (16,0) rectangle (25,1);
	\draw (26,0) -- ++(0,1) -- (30,1);
	\node[above] at (0.5,1) {$0$};
	\node[above] at (10.5,1) {$m$};
	\node[above] at (15.5,1) {$m+k$};
	\node[above] at (20.5,1) {$2m$};
	\node[above] at (25.5,1) {$2m+k$};
	
	\draw[<->] (10,-.5) to node[below] {$= m+A$} (15, -.5);
	\draw[<->] (20,-.5) to node[below] {$\supset 2m + (A+A)\cap[0,k]$} (25, -.5);
	\node[right] at (29.5, .5) {$\cdots$};

\end{tikzpicture}
\caption{Illustrating a numerical semigroup $\Lambda$ with multiplicity $m$ and type $(A; k)$. 
Here all the elements of $\Lambda$ lie in boxed regions. Shaded regions represent elements that must be in $\Lambda$.\label{fig:type-construction}}
\end{figure}

\begin{proof}
This lemma is a straightforward generalization of the procedure described in Example \ref{ex:type}. Indeed, if $\Lambda$ has multiplicity $m$ and type $(A; k)$, then we must have $\Lambda \cap [0, m+k] = \{0\} \cup (m + A)$, $2m+k \notin \Lambda$, and $[2m+k+1, \infty) \subset \Lambda$. Furthermore, because $\Lambda$ is closed under addition, and $m + A \subset \Lambda$, it follows that $2m + A + A \subset \Lambda$. There are no other restrictions on $\Lambda \cap [m+k+1, 2m+k-1]$, so every $\Lambda$ has the form \eqref{eq:constr2}. 

Every $\Lambda$ in \eqref{eq:constr2} is a numerical semigroup. Indeed, for any nonzero $a, b \in \Lambda$, we have $a, b \geq m$; now, either $a, b \leq m + k$, so that $a, b \in m + A$ and hence $a + b \in 2m + A + A \subset \Lambda$, or one of $a$ and $b$ is greater than $m + k$, in which case $a + b > 2m + k = f(\Lambda)$, and hence $a + b \in \Lambda$.
\end{proof}

As a corollary, we obtain the number of numerical semigroups of a given type and genus.

\begin{corollary} \label{cor:type-m-count}
Let $m, k,$ and $g$ be positive integers with $m > k$, and let $A \in \c A_k$. Then the collection of numerical semigroups with multiplicity $m$, type $(A; k)$, and genus $g$ is exactly the collection of sets of the form \eqref{eq:constr2}, where 
where $B$ is any subset of $[m+k+1, 2m+k-1] \setminus (2m + A + A)$ with $2m + k - \abs{A} - \abs{(A + A)\cap[0, k]} - g$ elements. Furthermore, the number of such numerical semigroups is equal to
\begin{equation} \label{eq:num-genus-type}
	\binom{m - 1 - \abs{(A + A)\cap[0, k]}}{g + \abs{A}-m-k - 1}.
\end{equation}
\end{corollary}

The last statement in the corollary follows from
\[
	\binom{m - 1 - \abs{(A + A)\cap[0, k]}}{2m + k - \abs{A} - \abs{(A + A)\cap[0, k]} - g}
	= \binom{m - 1 - \abs{(A + A)\cap[0, k]}}{g + \abs{A}-m-k - 1}.
\]

Now we can give the number of numerical semigroups of a given genus and type.

\begin{proposition} \label{prop:type-count}
Let $g$ and $k$ be positive integers, and let $A \in \c A_k$. Then the number of numerical semigroups with genus $g$ and type $(A; k)$ is at most $F_{g - \abs{(A + A)\cap[0, k]} + \abs{A} - k - 1}$, where equality holds if $3k \leq g + \abs{A} + \abs{(A + A)\cap[0, k]} -2$, which is true if $3k \leq g$.
\end{proposition}

\begin{proof}
Apply Corollary \ref{cor:type-m-count} and sum over all $m > k$, we see that the number of numerical semigroups with genus $g$ and type $(A; k)$ is exactly
\begin{equation} \label{eq:type-exact-count}
	\sum_{m > k} \binom{m - 1 - \abs{(A + A)\cap[0, k]}}{g + \abs{A}-m-k - 1}.
\end{equation}
If we relax the constraint $m > k$ in the sum, then we obtain the following upper bound to \eqref{eq:type-exact-count}
\begin{equation} \label{eq:type-count-ineq}
	\sum_{m} \binom{m - 1 - \abs{(A + A)\cap[0, k]}}{g + \abs{A}-m-k - 1} = F_{g - \abs{(A + A)\cap[0, k]} + \abs{A} - k - 1},
\end{equation}
where we again used \eqref{eq:fib}. This proves the first part of the proposition. For the equality case, we need to show that whenever $3k \leq g + \abs{A} + \abs{(A + A)\cap[0, k]} - 2$, the only positive terms in the left-hand side sum in \eqref{eq:type-count-ineq} are those with $m > k$. Indeed, the term corresponding to $m$ in the sum is zero unless
\[
	0 \leq g + \abs{A}-m-k - 1 \leq m - 1 - \abs{(A + A)\cap[0, k]},
\]
or equivalently,
\begin{equation} \label{eq:m-range}
	\frac 1 2 \( g + \abs{A} + \abs{(A + A)\cap[0, k]} - k \) \leq m \leq g + \abs{A} - k - 1.
\end{equation}
Suppose that $3k \leq g + \abs{A} + \abs{(A + A)\cap[0, k]} -2$ holds, or equivalently
\begin{equation} \label{eq:equal-cond}
	k+1 \leq \frac 1 2 \( g + \abs{A} + \abs{(A + A)\cap[0, k]} - k \).
\end{equation}
If the term in \eqref{eq:type-count-ineq} corresponding to $m$ is positive, then \eqref{eq:m-range} and \eqref{eq:equal-cond} would imply that
\[
	k+1 \leq \frac 1 2 \( g + \abs{A} + \abs{(A + A)\cap[0, k]} - k \)
	\leq m,
\]
so that $m > k$. Hence, $3k \leq g + \abs{A} + \abs{(A + A)\cap[0, k]} -2$ guarantees that the quantities in \eqref{eq:type-exact-count} and \eqref{eq:type-count-ineq} are equal, thereby establishing a sufficient equality condition.
\end{proof}

\subsection{Asymptotics}
Let $t_g$ denote the number of numerical semigroups $\Lambda$ of genus $g$ satisfying $f(\Lambda) < 3m(\Lambda)$. Recall that $n_g$ is the number of numerical semigroups of genus $g$, so $t_g \leq n_g$. Previously we found the number of numerical semigroups of a given genus and type. Now we shall sum over all the types to obtain bounds for $t_g$.

\begin{lemma} \label{lem:t_g-lower}
For any positive integer $g$, we have
\[
	t_g \geq F_{g+1} + \sum_{k = 1}^{\floor{g/3}} \sum_{A \in \c A_k} F_{g - \abs{(A + A)\cap[0, k]} + \abs{A} - k - 1}.
\]
\end{lemma}

\begin{proof}
From Proposition \ref{prop:f<2m} we know that the number of numerical semigroups of genus $g$ and satisfying $f < 2m$ is exactly $F_{g+1}$. Next we consider numerical semigroups of genus $g$ satisfying $2m < f < 3m$. Suppose that $1 \leq k \leq g/3$, and $A \in \c A_k$, then the equality condition of Proposition \ref{prop:type-count} is satisfied, so that the number of numerical semigroups of type $(A; k)$ is exactly $F_{g - \abs{(A + A)\cap[0, k]} + \abs{A} - k - 1}$. Now let us sum over all $(A; k)$ with $1 \leq k \leq g/3$, then we obtain that the number of numerical semigroups with $2m < f < 3m$ is at least
\[
	\sum_{k = 1}^{\floor{g/3}} \sum_{A \in \c A_k} F_{g - \abs{(A + A)\cap[0, k]} + \abs{A} - k - 1}.
\]
The lemma then follows immediately.
\end{proof}

Now we can deduce an asymptotic lower bound for $t_g$.

\begin{proposition} \label{prop:t_g-lower-asymp}
We have
\begin{equation} \label{eq:t_g-lower3}
	\liminf_{g \to \infty} t_g \vphi^{-g} 
	\geq 
	\frac{\vphi}{\sqrt 5} + 
	\frac{1}{\sqrt 5} \sum_{k = 1}^{\infty} \sum_{A \in \c A_k} \vphi^{ - \abs{(A + A)\cap[0, k]} + \abs{A} - k - 1}.
\end{equation}
\end{proposition}

\begin{proof}
First fix a positive integer $k_M$. Lemma \ref{lem:t_g-lower} implies that
\begin{equation} \label{eq:t_g-lower1}
	t_g \geq F_{g+1} + \sum_{k = 1}^{k_M} \sum_{A \in \c A_k} F_{g - \abs{(A + A)\cap[0, k]} + \abs{A} - k - 1}	
\end{equation}
for any $g \geq 3k_M$. Recall that $F_n = \frac{1}{\sqrt{5}} (\vphi^n - (-\vphi)^{-n})$. Multiplying both sides of \eqref{eq:t_g-lower1} by $\vphi^{-g}$ and then letting $g \to \infty$ (note that the right hand side remains a finite sum), we get
\begin{equation} \label{eq:t_g-lower2}
	\liminf_{g \to \infty} t_g \vphi^{-g} 
	\geq 
	\frac{\vphi}{\sqrt 5} + 
	\frac{1}{\sqrt 5} \sum_{k = 1}^{k_M} \sum_{A \in \c A_k} \vphi^{ - \abs{(A + A)\cap[0, k]} + \abs{A} - k - 1}.
\end{equation}
Since $k_M$ can be chosen to be arbitrarily large, we obtain \eqref{eq:t_g-lower3} by letting $k_M \to \infty$.
\end{proof}


See Table \ref{tab:sum} for some computed values for some partial sums for the right-hand side of \eqref{eq:t_g-lower3}. Based on numerical data and also some heuristic arguments, we strongly believe that this sum converges, though we currently do not have a proof.

\begin{conjecture} \label{conj:sum-converge}
The sum in the right-hand side of \eqref{eq:t_g-lower3} converges to a finite value.
\end{conjecture}

\begin{table}[ht!] \centering
\caption{Partial sums \eqref{eq:t_g-lower2} of the right-hand side (RHS) of \eqref{eq:t_g-lower3}.\label{tab:sum}}
\small
\begin{tabular}{cr@{.}lcr@{.}lcr@{.}lcr@{.}lcr@{.}l}
\toprule
$k_M$ & \multicolumn{2}{c}{RHS of \eqref{eq:t_g-lower2}} & 
$k_M$ & \multicolumn{2}{c}{RHS of \eqref{eq:t_g-lower2}} &
$k_M$ & \multicolumn{2}{c}{RHS of \eqref{eq:t_g-lower2}} &
$k_M$ & \multicolumn{2}{c}{RHS of \eqref{eq:t_g-lower2}} &
$k_M$ & \multicolumn{2}{c}{RHS of \eqref{eq:t_g-lower2}} \\
\midrule
0 & 0&72361 & 10 & 2&07121 & 20 & 3&02285 & 30 & 3&51068 & 40 & 3&72361 \\
1 & 0&89443 & 11 & 2&21950 & 21 & 3&10323 & 31 & 3&54849 & 41 & 3&73890 \\
2 & 1&00000 & 12 & 2&30278 & 22 & 3&15132 & 32 & 3&56996 & 42 & 3&74738 \\
3 & 1&17082 & 13 & 2&43780 & 23 & 3&22313 & 33 & 3&60088 & 43 & 3&76001 \\
4 & 1&27639 & 14 & 2&51719 & 24 & 3&26281 & 34 & 3&61913 & 44 & 3&76715 \\
5 & 1&45085 & 15 & 2&63446 & 25 & 3&32421 & 35 & 3&64565 & 45 & 3&77725 \\
6 & 1&55279 & 16 & 2&70447 & 26 & 3&35986 & 36 & 3&66030 & 46 & 3&78318 \\
7 & 1&72222 & 17 & 2&81245 & 27 & 3&41108 & 37 & 3&68251  \\
8 & 1&82191 & 18 & 2&87343 & 28 & 3&44105 & 38 & 3&69523  \\
9 & 1&97675 & 19 & 2&96852 & 29 & 3&48580 & 39 & 3&71321  \\
\bottomrule
\end{tabular}
\end{table}

Next we give an upper bound for $t_g$.

\begin{lemma} \label{lem:t_g-upper}
For any positive integer $g$, we have
\[
	t_g \leq F_{g+1} + \sum_{k = 1}^{g-1} \sum_{A \in \c A_k} F_{g - \abs{(A + A)\cap[0, k]} + \abs{A} - k - 1}.
\]
\end{lemma}

\begin{proof}
The argument is essentially the same as that of Lemma \ref{lem:t_g-lower}, except that we now use the upper bound result in Proposition \ref{prop:type-count}. The first term on the right-hand side is the number of numerical semigroups with $f < 2m$. The sum is, by Proposition \ref{prop:type-count}, an upper bound to the number of numerical semigroups with $2m < f < 3m$. Note that we only need to sum up to $k = g-1$, since for $f \geq g$, we have $g - \abs{(A + A)\cap[0, k]} + \abs{A} - k - 1 \leq 0$ so that the term $F_{g - \abs{(A + A)\cap[0, k]} + \abs{A} - k - 1}$ is zero.
\end{proof}

\begin{lemma} \label{lem:t_g-upper2}
We have
\begin{equation} \label{eq:t_g-upper2}
	\limsup_{g \to \infty} t_g \vphi^{-g} 
	\leq 
	\frac{\vphi}{\sqrt 5} + 
	\frac{1}{\sqrt 5} \sum_{k = 1}^{\infty} \sum_{A \in \c A_k} \vphi^{ - \abs{(A + A)\cap[0, k]} + \abs{A} - k - 1}.
\end{equation}
\end{lemma}

\begin{proof}
If Conjecture \ref{conj:sum-converge} were false and the right-hand side of \eqref{eq:t_g-upper2} diverges to infinity, and then the lemma is vacuously true. So assume that the sum converges. Fix a positive integer $k_M$. From Lemma \ref{lem:t_g-upper} we obtain that
\begin{align*}
	t_g	
&	\leq F_{g+1} + \sum_{k = 1}^{g - 1} \sum_{A \in \c A_k} F_{g - \abs{(A + A)\cap[0, k]} + \abs{A} - k - 1}
\\&	= F_{g+1} + \sum_{k = 1}^{k_M} \sum_{A \in \c A_k} F_{g - \abs{(A + A)\cap[0, k]} + \abs{A} - k - 1} 
		+ \sum_{k > k_M} \sum_{A \in \c A_k} F_{g - \abs{(A + A)\cap[0, k]} + \abs{A} - k - 1}.
\end{align*}
Now multiply both sides by $\vphi^{-g}$ and let $g \to \infty$. We have
\[
	\lim_{g \to \infty} \vphi^{-g} \sum_{k = 1}^{k_M} \sum_{A \in \c A_k} F_{g - \abs{(A + A)\cap[0, k]} + \abs{A} - k - 1} 
	= \frac{1}{\sqrt 5} \sum_{k = 1}^{k_M} \sum_{A \in \c A_k} \vphi^{ - \abs{(A + A)\cap[0, k]} + \abs{A} - k - 1}
\]
since the number of terms in the sum is always finite, and also
\[
	\sum_{k > k_M} \sum_{A \in \c A_k} F_{g - \abs{(A + A)\cap[0, k]} + \abs{A} - k - 1}
	\leq \frac{2}{\sqrt{5}} \sum_{k > k_M} \sum_{A \in \c A_k} \vphi^{g - \abs{(A + A)\cap[0, k]} + \abs{A} - k - 1}
\]
since $F_n \leq \frac{2}{\sqrt{5}}\vphi^n$. Combining the two statements gives us
\begin{multline} \label{eq:t_g-upper3}
	\limsup_{g \to \infty} t_g\vphi^{-g}
	\leq \frac{\vphi}{\sqrt 5} + 
	   \frac{1}{\sqrt 5} \sum_{k = 1}^{k_M} \sum_{A \in \c A_k} \vphi^{ - \abs{(A + A)\cap[0, k]} + \abs{A} - k - 1} \\
	   +  \frac{2}{\sqrt 5} \sum_{k > k_M} \sum_{A \in \c A_k} \vphi^{ - \abs{(A + A)\cap[0, k]} + \abs{A} - k - 1}.
\end{multline}
This is true for arbitrarily large $k_M$. Since we are in the case where we assume Conjecture \ref{conj:sum-converge}, the third term in right-hand side of \eqref{eq:t_g-upper3} goes to zero as $k_M \to \infty$. Therefore, letting $k_M \to \infty$, \eqref{eq:t_g-upper3} implies that
\[
	\limsup_{g \to \infty} t_g\vphi^{-g} \leq \frac{\vphi}{\sqrt 5} + 
	\frac{1}{\sqrt 5} \sum_{k = 1}^{\infty} \sum_{A \in \c A_k} \vphi^{ - \abs{(A + A)\cap[0, k]} + \abs{A} - k - 1}. \qedhere
\]
\end{proof}

Combining Proposition \ref{prop:t_g-lower-asymp} and Lemma \ref{lem:t_g-upper2}, we obtain our main result, which gives the asymptotics for $t_g$.

\begin{theorem} \label{thm:t_g}
Let $t_g$ denote the number of numerical semigroups $\Lambda$ of genus $g$ that satisfy $f(\Lambda) < 3m(\Lambda)$. Then
\[
	\lim_{g \to \infty} t_g \vphi^{-g} 
	=
	\frac{\vphi}{\sqrt 5} + 
	\frac{1}{\sqrt 5} \sum_{k = 1}^{\infty} \sum_{A \in \c A_k} \vphi^{ - \abs{(A + A)\cap[0, k]} + \abs{A} - k - 1}.
\]
\end{theorem}

Using the computed values in Table \ref{tab:sum}, we obtain the following result about the number of numerical semigroups of a given genus.

\begin{proposition} \label{prop:numerical-bound}
Let $n_g$ be the number of semigroups of genus $g$, and $t_g$ the number of numerical semigroups of genus $g$ satisfying $f < 3m$. Then
\[
	 \liminf_{g \to \infty} n_g \vphi^{-g} 
			\geq \lim_{g \to \infty} t_g \vphi^{-g}
			> 3.78.
\]
\end{proposition}

\begin{proof}
The first inequality follows from $n_g \geq t_g$, which is true by definition. For the second inequality, apply \eqref{eq:t_g-lower2} with $k_M = 46$ (see Table \ref{tab:sum} for computed values of the sum).
\end{proof}

\begin{remark}
The quantity $k_M$ used throughout this section can be viewed as a computation parameter that is unrelated to the genus $g$. We can obtain better numerical lower bounds by increasing $k_M$, though the cost of computation increases exponentially fast with $k_M$. In particular, the convergence rate of the right hand side of \eqref{eq:t_g-lower2} as $k_M \to \infty$ is unrelated to the convergence rate of $t_g \vphi^{-g}$ as $g \to \infty$. In fact, we do not even know whether $t_g \vphi^{-g}$ is increasing. In other words, our approach analyzes $n_g \vphi^{-g}$ indirectly indexing over the types as opposed to the genus.  The plot in Figure \ref{fig:ratio-plot} illustrates that our approach for computing a lower bound to $\liminf n_g \vphi^{-g}$  ``sees further'' than the actual computed values of $n_g$ or $t_g$.
\end{remark}

\section{Numerical semigroups with $f > 3m$} \label{sec:f>3m}

The methods presented in the previous sections do not readily extend to analyzing the numerical semigroups $\Lambda$ with $f(\Lambda) > 3 m(\Lambda)$. We were unable to construct large families (i.e., with cardinality at least $c\vphi^g$ for some constant $c$) of such numerical semigroups as we did in Sections \ref{sec:f<2m} and \ref{sec:f<3m}. See Table \ref{tab:t_g} and Figure \ref{fig:ratio-plot2} for some data on the proportion $t_g/n_g$ of numerical semigroups of a given genus satisfying $f < 3m$. Based on naive extrapolation of the data assuming that $t_g/n_g$ has roughly geometrically decreasing increments, we conjecture that $t_g/n_g$ approaches $1$ as $g \to \infty$. In other words, we believe that the numerical semigroups satisfying $f > 3m$ occupy an asymptotically negligible proportion of all numerical semigroups of a given genus $g$, for $g$ large.

\begin{conjecture} \label{conj:f>3m}
Let $n_g$ be the number of numerical semigroups of genus $g$, and $t_g$ the number of numerical semigroups $\Lambda$ of genus $g$ satisfying $f(\Lambda) < 3m(\Lambda)$. Then $t_g/n_g \to 1$ as $g \to \infty$.
\end{conjecture}

Conjecture \ref{conj:sum-converge} would imply that $\lim t_g \vphi^{-g}$ is finite, and Conjecture \ref{conj:f>3m} would imply that $\lim n_g \vphi^{-g} = \lim t_g \vphi^{-g}$, so the two Conjectures together would imply Conjecture \ref{conj:ratio} and hence also Conjecture \ref{conj:B}. In fact, we only need a weaker form of Conjecture \ref{conj:f>3m} saying that $t_g/n_g$ approaches some positive limit.

\begin{table}[ht!]
\caption{Some data on the number $n_g$ of numerical semigroups of genus $g$, and the number $t_g$ of numerical semigroups of genus $g$ satisfying $f < 3m$. \label{tab:t_g}}
\footnotesize
\begin{tabular}{rrrlllc@{\qquad}rrrlll}
\toprule
\multicolumn{1}{c}{$g$} & \multicolumn{1}{c}{$n_g$} & \multicolumn{1}{c}{$t_g$} & \multicolumn{1}{c}{$n_g\vphi^{-g}$} &  \multicolumn{1}{c}{$t_g\vphi^{-g}$} & \multicolumn{1}{c}{$t_g/n_g$} &  &
\multicolumn{1}{c}{$g$} & \multicolumn{1}{c}{$n_g$} & \multicolumn{1}{c}{$t_g$} & \multicolumn{1}{c}{$n_g\vphi^{-g}$} &  \multicolumn{1}{c}{$t_g\vphi^{-g}$} & \multicolumn{1}{c}{$t_g/n_g$} \\
\midrule
1 & 1 & 1 & 0.61803 & 0.61803 & 1.00000 &  & 26 & 770832 & 653420 & 2.83976 & 2.40721 & 0.84768 \\
2 & 2 & 2 & 0.76393 & 0.76393 & 1.00000 &  & 27 & 1270267 & 1080981 & 2.89220 & 2.46123 & 0.85099 \\
3 & 4 & 4 & 0.94427 & 0.94427 & 1.00000 &  & 28 & 2091030 & 1786328 & 2.94243 & 2.51366 & 0.85428 \\
4 & 7 & 6 & 1.02129 & 0.87539 & 0.85714 &  & 29 & 3437839 & 2948836 & 2.98981 & 2.56454 & 0.85776 \\
5 & 12 & 11 & 1.08204 & 0.99187 & 0.91667 &  & 30 & 5646773 & 4863266 & 3.03509 & 2.61396 & 0.86125 \\
6 & 23 & 20 & 1.28175 & 1.11456 & 0.86957 &  & 31 & 9266788 & 8013802 & 3.07831 & 2.66208 & 0.86479 \\
7 & 39 & 33 & 1.34323 & 1.13658 & 0.84615 &  & 32 & 15195070 & 13194529 & 3.11960 & 2.70888 & 0.86834 \\
8 & 67 & 57 & 1.42618 & 1.21332 & 0.85075 &  & 33 & 24896206 & 21707242 & 3.15894 & 2.75431 & 0.87191 \\
9 & 118 & 99 & 1.55236 & 1.30241 & 0.83898 &  & 34 & 40761087 & 35684639 & 3.19644 & 2.79835 & 0.87546 \\
10 & 204 & 168 & 1.65865 & 1.36594 & 0.82353 &  & 35 & 66687201 & 58618136 & 3.23203 & 2.84096 & 0.87900 \\
11 & 343 & 287 & 1.72357 & 1.44217 & 0.83673 &  & 36 & 109032500 & 96221845 & 3.26589 & 2.88216 & 0.88251 \\
12 & 592 & 487 & 1.83853 & 1.51244 & 0.82264 &  & 37 & 178158289 & 157840886 & 3.29810 & 2.92198 & 0.88596 \\
13 & 1001 & 824 & 1.92130 & 1.58157 & 0.82318 &  & 38 & 290939807 & 258749944 & 3.32869 & 2.96040 & 0.88936 \\
14 & 1693 & 1395 & 2.00831 & 1.65481 & 0.82398 &  & 39 & 474851445 & 423906805 & 3.35768 & 2.99745 & 0.89271 \\
15 & 2857 & 2351 & 2.09457 & 1.72361 & 0.82289 &  & 40 & 774614284 & 694076610 & 3.38517 & 3.03321 & 0.89603 \\
16 & 4806 & 3954 & 2.17762 & 1.79157 & 0.82272 &  & 41 & 1262992840 & 1135816798 & 3.41120 & 3.06772 & 0.89931 \\
17 & 8045 & 6636 & 2.25287 & 1.85830 & 0.82486 &  & 42 & 2058356522 & 1857750672 & 3.43589 & 3.10103 & 0.90254 \\
18 & 13467 & 11116 & 2.33074 & 1.92385 & 0.82543 &  & 43 & 3353191846 & 3037078893 & 3.45931 & 3.13320 & 0.90573 \\
19 & 22464 & 18593 & 2.40282 & 1.98877 & 0.82768 &  & 44 & 5460401576 & 4962738376 & 3.48152 & 3.16421 & 0.90886 \\
20 & 37396 & 31042 & 2.47214 & 2.05209 & 0.83009 &  & 45 & 8888486816 & 8105674930 & 3.50255 & 3.19408 & 0.91193 \\
21 & 62194 & 51780 & 2.54102 & 2.11554 & 0.83256 &  & 46 & 14463633648 & 13233250642 & 3.52246 & 3.22281 & 0.91493 \\
22 & 103246 & 86223 & 2.60702 & 2.17718 & 0.83512 &  & 47 & 23527845502 & 21595419304 & 3.54130 & 3.25044 & 0.91787 \\
23 & 170963 & 143317 & 2.66800 & 2.23657 & 0.83829 &  & 48 & 38260496374 & 35227607540 & 3.55913 & 3.27700 & 0.92073 \\
24 & 282828 & 237936 & 2.72784 & 2.29486 & 0.84127 &  & 49 & 62200036752 & 57443335681 & 3.57599 & 3.30252 & 0.92353 \\
25 & 467224 & 394532 & 2.78506 & 2.35175 & 0.84442 &  & 50 & 101090300128 & 93635242237 & 3.59193 & 3.32703 & 0.92625 \\
\bottomrule
\end{tabular}
\end{table}

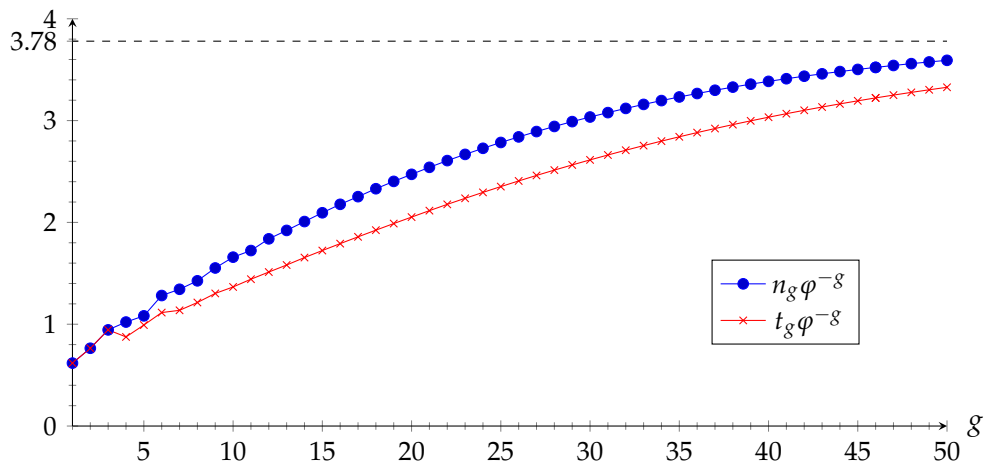
\begin{figure}[ht!] \centering
\begin{tikzpicture}
\begin{axis}[width=.8\textwidth, height = 7cm, enlargelimits=false,
			 xmin = 1, ymin = 0, ymax=4, minor tick num = 4,
		     axis x line = bottom, axis y line = left,
		     extra y ticks={3.78},
			 extra y tick labels={$3.78$},
			 legend style={ at={(.9,.2)}, anchor=south east},
			 font = \small,
			 name=plot]
	\addplot table[x=g,y=A] {plotdata.txt};
	\addlegendentry{$n_g \vphi^{-g}$}
	\addplot[mark=x, color=red] table[x=g,y=B] {plotdata.txt};
	\addlegendentry{$t_g \vphi^{-g}$}
	\addplot[dashed] coordinates { (0,3.78) (50, 3.78)};
\end{axis}

\node[label=right:$g$] at (plot.right of origin) {};

\end{tikzpicture}
\caption{Plot of $n_g \vphi^{-g}$ and $t_g \vphi^{-g}$ from Table \ref{tab:t_g}. Proposition \ref{prop:numerical-bound} shows that $3.78$ is an eventual \emph{lower bound} to both sequences.\label{fig:ratio-plot}}
\end{figure}

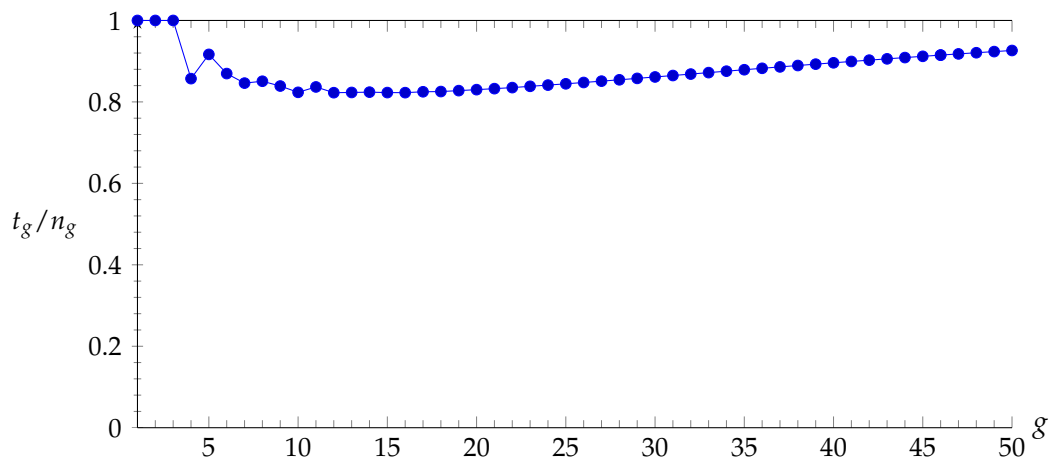
\begin{figure}[ht!] \centering
\begin{tikzpicture}
\begin{axis}[width=.8\textwidth, height = 7cm, enlargelimits=false,
			 xmin = 1, ymin = 0, ymax=1, minor tick num = 4,
			 ylabel = \rotatebox{-90}{$t_g/n_g$},
		     axis x line = box, axis y line = left,
		     extra y ticks={3.78},
			 extra y tick labels={$3.78$},
			 legend style={ at={(.9,.2)}, anchor=south east},
			 font = \small,
			 name=plot]
	\addplot table[x=g,y=R] {plotdata.txt};
\end{axis}

\node[label=right:$g$] at (plot.right of origin) {};

\end{tikzpicture}
\caption{Plot of $t_g/n_g$ from Table \ref{tab:t_g}.\label{fig:ratio-plot2}}
\end{figure}

\section*{Acknowledgments}

The author would like to thank Nathan Kaplan for introducing him to the problem and for helpful discussions, Joseph Gallian for proofreading the paper, and Maria Bras-Amor{\'o}s for providing the data in Table \ref{tab:t_g}.

\bibliographystyle{amsplain}
\bibliography{ref-semigroup}

\end{document}